\documentclass[12pt,leqno]{amsart}

\usepackage{amsmath,amssymb,amscd,amsthm}
\usepackage{latexsym}
\usepackage[dvips]{graphicx,epsfig}
\usepackage{graphicx}
\usepackage{color}

\newtheorem{theorem}{Theorem}
\newtheorem{lemma}{Lemma}
\newtheorem{proposition}{Proposition}

\newtheorem{corollary}{Corollary}

\newtheorem{claim}{Claim}

\newcommand{\f}[2]{\frac{#1}{#2}}
\newcommand{\dpr}[2]{\langle #1,#2 \rangle}

\newcommand{\al}{\alpha}

\newcommand{\ga}{\gamma}

\newcommand{\de}{\delta}

\newcommand{\la}{\lambda}

\newcommand{\si}{\sigma}

\newcommand{\rone}{\mathbf R}

\newcommand{\eps}{\epsilon}

\newcommand{\ca}{\mathcal A}

\newcommand{\p}{\partial}

\newcommand{\beq}{\begin{equation}}
\newcommand{\eeq}{\end{equation}}
\newcommand{\beqna}{\begin{eqnarray*}}
\newcommand{\eeqna}{\end{eqnarray*}}
\newcommand{\beqn}{\begin{equation*}}
\newcommand{\eeqn}{\end{equation*}}
\newcommand{\bp}{\begin{proof}}
\newcommand{\ep}{\end{proof}}
\newcommand{\bprop}{\begin{proposition}}
\newcommand{\eprop}{\end{proposition}}
\newcommand{\bt}{\begin{theorem}}
\newcommand{\et}{\end{theorem}}
\newcommand{\bex}{\begin{Example}}
\newcommand{\eex}{\end{Example}}
\newcommand{\bc}{\begin{corollary}}
\newcommand{\ec}{\end{corollary}}
\newcommand{\bcl}{\begin{claim}}
\newcommand{\ecl}{\end{claim}}
\newcommand{\bl}{\begin{lemma}}
\newcommand{\el}{\end{lemma}}

\begin{document}

\title
[Damped fractional Klein-Gordon equation]
{On the energy decay rates for the 1D damped fractional Klein-Gordon equation}

\author{Satbir Malhi, Milena Stanislavova}

 
 \address{Satbir Malhi, Department of Mathematics, University of Kansas,
 	1460 Jayhawk
 	Boulevard,  Lawrence KS 66045--7523}
 \email{smalhi@ku.edu}

\address{Milena Stanislavova, Department of Mathematics, University of Kansas,
1460 Jayhawk
Boulevard,  Lawrence KS 66045--7523}
\email{stanis@ku.edu}

\thanks{   
Stanislavova is partially supported by NSF-DMS, Applied Mathematics 
grant \# 1516245.}

\date{\today}


\begin{abstract}
 We consider the fractional Klein-Gordon equation in one spatial dimension, subjected to a damping coefficient, which  is non-trivial and periodic, or more generally strictly positive on a periodic set. We show that the energy of the solution  decays at the  polynomial rate $O(t^{-\f{s}{4-2s}})$ for $0< s<2 $ and at some exponential rate when $s\geq 2$. Our  approach  is based on the asymptotic theory of $C_0$ semigroups in which one can  relate the  decay rate of the energy in terms of the resolvent growth of the semigroup generator. The main technical result is a new  observability estimate for the fractional Laplacian, which may be of independent interest. 
  
\end{abstract}
\maketitle
\section{Introduction}
In this paper, we  consider the energy decay of the following fractional damped Klein-Gordon equation
\begin{eqnarray}\label{eq1.1}
u_{tt}+\gamma(x)u_t+(-\p_{xx})^{s/2}u+m u=0, \ \ (t,x)\in \rone_+\times \rone, 
\end{eqnarray}
where $m>0$ and $\ga(x)\geq 0$ is bounded below by a positive constant on a $ 2\pi$-periodic  set. The parameter $\textit{s}$ refers to the fractional order of the spatial derivative and describes the fractional nature of the equation. Here and throughout,  $u(x,t)$ is generally a complex-valued function, and the pseudo-differential operator $(-\p_{xx})^{s/2}$ is defined through its Fourier multiplier 
$$\widehat{(-\p_{xx})^{s/2} f}(\xi) = |\xi|^s \hat{f}(\xi), ~\xi\in\mathbb{R}.$$
The function $\gamma(x)$ denotes the damping force, which travels with velocity $u_t$ and causes the loss of energy decay in the system. This energy decay is the main object of study in this article.

For the case $s=2$, the operator $-\p_{xx}$ denotes the positive Laplacian. In this case, (\ref{eq1.1})  reduces to the well know classical Damped Klein-Gordon equation. It has been studied extensively in the last decade by many authors. On bounded domains,  Bardos, Lebeau, Rauch, and Taylor in \cite{bardos1988exemple,bardos1992sharp,rauch1974exponential} proved exponential energy decay rate under the geometric control condition (GCC) in a sense that there exist $T,~\epsilon>0,$ such that $\int_0^T\gamma(x(t))~dt\geq \epsilon$ along every straight line unit speed trajectory. Whereas recently,  Burq and Joly \cite{doi:10.1142/S0219199716500127} extended these results to non-compact setting assuming additional smoothness on $\gamma(x)\in \mathcal{C}^{\infty}$. In the absence of GCC, several authors proved polynomial rate of decay in different setting (see\cite{anantharaman2012decay,Hitrik,stahn2017optimal,phung2007polynomial,wunsch2015periodic} and reference therein).

In fact for the case $0<s<2$ or $s>2$, to these authors knowledge there has been no rigorous study of the energy decay rate of the  damped Klein-Gordon equation in the fractional case. It  is our goal here to compute   the decay rate of the solution $u(x,t)$ of  \eqref{eq1.1} in terms of its fractional power $s$, while  the data is  smoother than the original energy space. This is  achieved under the assumption that  $\ga(x)$ is non-trivial and periodic, or more generally strictly positive on a periodic set.

We show that for low order fractional power $0<s<2$, the rate of decay is algebraic. This is in sharp contrast with the case $s\geq 2$,  where the solution has exponential rate of decay. So, it appears that $s=2$ is exactly a threshold value, which separates the algebraic from exponential rate of decay, but unfortunately our method does not address the optimality of this exponent. This remains an open question for future investigations. 

The main result of the paper is as follow.
\begin{theorem}\label{thm1.1}
	Let   $m>0$ and $0\leq \ga(x)\in L^{\infty}$ and that there exist $\epsilon >0$ and a $2\pi\mathbb{Z}$- invariant open set $\Omega\subset\mathbb{R}$ such that $\gamma(x)\geq \epsilon$ for a.e. $x\in \Omega$. Then there exists $C>0$ so that 
\begin{itemize}
	\item for $0<s<2$, we have
	\begin{equation}
	\label{100}
		\|(u(t), u_t(t))\|_{H^{s/2}\times L^2}\leq \frac{C}{1+t^{\f{s}{4-2s}}}\|(u(0), u_t(0))\|_{H^{s}\times H^{s/2}}.
	\end{equation} 

 \item for $s\geq 2$, there exists $\la_0>0$, so that 
	\begin{equation}
	\label{s2}
		\|(u(t), u_t(t))\|_{H^{s/2}\times L^2}\leq Ce^{-\lambda_0 t}\|(u(0), u_t(0))\|_{H^{s/2}\times L^2}.
	\end{equation} 
\end{itemize}	
\end{theorem}
The proof of Theorem \ref{thm1.1} is based on the semigroup technique used in \cite{wunsch2015periodic,Hitrik,milena,Latushkin}, in which rather than estimating  norm of the solution directly, we used the following two classical results. Gearhart-Pr\"uss Theorem \cite{gearhart1978spectral,prss1984spectrum} and Borichev-Tomilov Theorem  in \cite{yuri} make it possible to deduce sharp rates of energy decay from appropriate growth bounds on the norm of the resolvent of  the semigroup`s generator. 

Let us state precisely these two results, namely  Gearhart-Pr\"uss and Borichev-Tomilov theorems, 
which allow us to compute the rates specified in Theorem \ref{thm1.1}. 
The Gearhart-Pr\"uss theorem provides a necessary and sufficient 
criteria for negative growth bounds for a given  semigroup, in terms of the natural spectral condition $i\mathbb{R} \subset \rho(\ca)$ and appropriate resolvent bounds for the generator. We use here the characterization of  Gearhart-Pr\"uss theorem given by Huang in \cite{huang1985characteristic}.
\begin{theorem}[Gearhart-Pr\"uss]
	\label{Gearhart}
	
	Let $e^{t\ca }$ be a $C_0$-semigroup in a Hilbert space $H$ and assume that there exists a positive constant $M>0$ such that $\|e^{t\ca }\|\leq M$ for all $t\geq 0$. Let $\mu\in \rho(\ca)$. Then the following are equivalent. 
	\begin{itemize}
		\item There exists $\la_0>0$ and $C$, so that 
		$$
		\|T(t)\|_{B(H)}\leq C e^{-\la_0 t}
		$$
		\item $i\mathbb{R} \subset \rho(\ca)$ and 
		\begin{eqnarray}\label{eqn8.1.3}
		\sup\limits_{k\in\mathbb{R}}\|(\ca-i kI)^{-1}\|_{B(H)}<+\infty.
		\end{eqnarray}
	\end{itemize}
	
\end{theorem}
For semigroups lacking the aforementioned resolvent bounds, but still satisfying the natural spectral condition $i\mathbb{R} \subset \rho(\ca)$, one can still establish algebraic rates, by showing that the resolvent satisfies a  power growth. 
\begin{theorem}[Borichev-Tomilov]
	\label{mainthm1}
	Let $T(t)$ be a bounded $C_0$-semigroup on a Hilbert space $H$ with generator $A$ such that $i\mathbb{R}\subset \rho(A)$. Then for a fixed $\alpha >0$, 
	\begin{align*}
	\|R(i k,A)\|=O(|k|^\alpha),~|k| \longrightarrow\infty.
	\end{align*}
	implies
	\begin{align*}
	\|T(t)A^{-1}\|=O\left(\frac{1}{t^{1/\alpha}}\right),~t\longrightarrow\infty.
	\end{align*}
\end{theorem}

The paper is set out as follows. 
In Section 2, we prove the observability estimate for the fractional Laplacian. Using this, we derive a resolvent estimate for our damped problem.  In Section 3, we turn to the main resolvent bounds. We first establish an upper bound for  norm of the resolvent operator along the imaginary axis via  the fractional observability estimate. 
Later, at the end of section 3 , we apply the Gearhart-Pr\"uss Theorem and  Borichev-Tomilov results respectively to deduce from these resolvent bounds an estimate for the rate of energy decay of smooth solutions.

\section{Observability Estimates} 
We start with a few preliminary notations. 
\subsection{Function spaces, Fourier transforms, symbols} 
The spaces $L^p(\rone), 1\leq p\leq \infty$ are defined in a standard way. The Fourier transform for us will be given by 
$$
\hat{f}(\xi)=\f{1}{\sqrt{2\pi}} \int_{-\infty}^{\infty} f(x) e^{- i x \xi} dx, \ \ f(x)=\f{1}{\sqrt{2\pi}} 
\int_{-\infty}^\infty \hat{f}(\xi) e^{ i x \xi} d\xi. 
$$
The operator $-\p_{xx}$ can be realized as $\widehat{-\p_{xx} f}(\xi)=  \xi^2 \hat{f}(\xi)$. For any $s>0$, one can write $\widehat{(-\p_{xx})^{s/2} f}(\xi) = |\xi|^s \hat{f}(\xi)$. 

The fractional Sobolev spaces $H^s(\rone)$ can be identified as the   set of all functions $f$, so that $[(-\p_{xx})^{s/2}+1]f\in L^2(\rone)$. Alternatively, the norm is defined as follows 
$$
\|f\|_{H^s(\rone)}^2= \int_{-\infty}^\infty (1+\xi^2)^s |\hat{f}(\xi)|^2 d\xi<\infty.
$$
For periodic functions defined on  $[-1,1]$, which are sufficiently smooth, there is the usual Fourier series representation 
$$
f=\sum_k f_k e^{ik \pi x}, f_k=\f{1}{2} \int_{-1}^1 f(x) e^{- i k \pi  x} dx, 
$$
with $\|f\|_{L^2[-1,1]}^2= 2 \sum_k |f_k|^2$.  
The fractional operator $(-\p_{xx})^{s/2}$ using functional calculus is defined through 
\[
(-\p_{xx})^{s/2}f=\sum_{k=-\infty}^{\infty}(\pi |k|)^{s} f_ke^{ i k \pi x},
\]
for sufficiently smooth functions $f\in L^2[-1,1]$. 

The observability estimate for $s=2$ has been  proved by Burq and Zworski  in \cite{burq2005bouncing} on a two-dimensional compact manifold. Recently,  Wunsch \cite{wunsch2015periodic} extended these estimates to $\mathbb{R}^n$ under a periodic setting. In this note we prove similar observability estimate for the fractional case. In fact, in  the case of one-dimension our estimate contains an additional decay factor, which helps us to improve Wunsch`s results in the one-dimensional setting.

\subsection{Main observability lemma for the fractional Laplacian}
The following estimate, which may be of interest in its own right, gives $L^2$ control of the resolvent of the free Laplacian on its spectra, modulo an error term. 
\begin{theorem}
	\label{theo:2}
	Let $s>0$, $\lambda\geq 1$ and $\Omega\subset \mathbb{R}$ be a non-empty, $2\pi \mathbb{Z}$ invariant open set. For all
	$\la\in \mathbb{R}$, let  $
	((-\p_{xx})^{s/2}-\lambda) u = f$.
	Then, there exists $C$, so that 
	\begin{equation}\label{eqn:20}
	\|u\|_{L^2}\leq C( <\lambda>^{\f{1}{s}-1} \|f\|_{L^2} +\|u\|_{L^2(\Omega)}).
	\end{equation} 
\end{theorem}
Let us explain the idea behind such result. Clearly, the difficult case is when $\la>0$ and large. Since the spectrum, $\si((-\p_{xx})^{s/2})=\si_{a.c.}((-\p_{xx})^{s/2})=[0, \infty)$, we cannot expect 
$[(-\p_{xx})^{s/2}-\la]^{-1}$ to be bounded on $L^2$, and it is not. Instead, \eqref{eqn:20} asserts that such an $L^2$ resolvent bound almost holds (with an additional decay rate of $\la^{\f{1}{s}-1}$, which is important for our purposes), modulo an extra ``control'' term.

The method of proof is to first establish the above estimate on the bounded interval $[-1,1]$. We then use the technique of Wunsch, \cite{wunsch2015periodic}   to extend the result to the real line $\mathbb{R}$.
\subsubsection{Observability on intervals}

 We start with an elementary lemma.
		\begin{lemma}
			\label{elem} 
			Let $s>0$. Then, there exists $d_s, D_s$, so that for every $0<x<y$
	\begin{equation}
	\label{obv} 
		d_s\max(x,y)^{s-1} |x-y|\leq |x^s-y^s|\leq 	D_s\max(x,y)^{s-1} |x-y|.
	\end{equation}
		\end{lemma}
		\begin{proof}
		Start with the function $f_s(z)=\frac{1-z^s}{1-z}$, defined for $z\in [0,1]$. Clearly this is a continuous function on $[0,1]$ (defined at $z=1$ via $f(1)=s$), so it has a minimum and maximum, say $d_s, D_s$. That is, 
		$$
		d_s(1-z)\leq 1-z^s\leq D_s(1-z).
		$$
		Without loss of generality $x\leq y$ and apply the previous inequality to $z=\frac{x}{y}$. This shows \eqref{obv}. 
		\end{proof}

\begin{lemma}\label{theo:10}
Let $s>0$. 	Consider the following damped fractional Laplace equation on $[-1,1]$
	\begin{eqnarray}\label{10}
	((-\p_{xx})^{s/2}-\la )u = f, x\in [-1,1]. 
	\end{eqnarray}
	Then for every $\de>0$ there is $C_\de$ so that  
	\begin{eqnarray}\label{20}
	\|u\|_{L^2[-1,1]}\leq C_\de[ <\la>^{\frac{1}{s}-1} \|f\|_{L^2[-1,1]}+ \|u\|_{L^2[-\de,\de]}]
	\end{eqnarray}
	for solutions $u$ of \eqref{10}, where $<\la>:=(1+|\la|^2)^{1/2}$. 
\end{lemma}
\begin{proof}
	We can always assume that $u,f$ are real, otherwise split in real and imaginary parts. We split the argument in the cases where $f$ is an even function ( in which case $u$ is also even function ) and then when $f$ is an odd function ($u$ odd respectively). \\
{\bf Case I: $u, f$   are  even functions:} 
	For $u, f$ even, we can expend $u$ and $f$ in cosine  series as follows 
	\[
	u=\sum_{k=0}^\infty u_k \cos(k\pi x),~ f=\sum_{k=0}^\infty f_k \cos(k\pi x)
	\]
	In this case, 
	$$
	(-\p_{xx})^{s/2} u(x) = \sum_{k=0}^\infty (\pi k)^s u_k \cos(k\pi x),
	$$
	Assume first that $\la=-\pi^s \si^s, \si>\f{1}{2}$. Then, taking a dot product with $u$ in \eqref{10}, we have 
	$$
	-\la \|u\|^2<\|(-\partial_{xx})^{s/4}u\|^2-\la \|u\|^2=\dpr{f}{u}\leq -\f{\la}{2} \|u\|^2+ \f{C}{|\la|} \|f\|^2
	$$
	Thus, we have better estimate in this case
	\begin{eqnarray}\label{eqn3}
	\|u\|_{L^2}\leq \frac{C}{|\la|}\|f\|_{L^2}
	\end{eqnarray}
	
	Next, let us take $\la=\pi^s \si^s, \si>\f{1}{2}$. Let $k_0=]\si[$, that is, the closest integer to $\si$ using the smaller integer when $\si$ is a half number. 
	Then for every $k\neq k_0$,  we have 
	\begin{equation}
	u_k=\f{1}{\pi^s(k^s-\si^s)} f_k, k\neq k_0.
	\end{equation}
We wish to estimate the function 
$$
\tilde{u}=\sum_{k\neq k_0} u_k \cos(\pi  k x)=u-u_{k_0} \cos(\pi  k_0 x)
$$
first. 
By Lemma \ref{elem}, we have that $ |k^s-\si^s|\sim  |k-\si| \max(k, \si)^{s-1}, k\neq k_0$. \\
{\bf Case I: $s\geq 1$} In this case, we can further take $ |k^s-\si^s|\geq C  |k-\si| \si^{s-1}, k\neq k_0$.  
We have 
	\begin{eqnarray*}
		\|\tilde{u}\|_{L^2}^2=\sum_{k\neq k_0, k\geq 0} |u_k|^2\leq \frac{1}{\pi^{2s}\si^{2s-2}}\sum_{k\neq k_0, k\geq 0}  \f{C}{|k-\si|^2} f_k^2
		\leq \f{C}{\pi^{2s}\si^{2s-2}} \|f\|^2=\f{C}{\lambda^{2-\frac{2}{s}}} \|f\|^2.
	\end{eqnarray*}
	Thus,
	\begin{eqnarray}\label{eq4}
	\|\tilde{u}\|_{L^2}\leq C <\lambda>^{\frac{1}{s}-1} \|f\|_{L^2}
	\end{eqnarray}
	{\bf Case II: $0<s<1$}  In this case, we have 
		\begin{eqnarray*}
			\|\tilde{u}\|_{L^2}^2=\sum_{k\neq k_0, k\geq 0} |u_k|^2\leq \frac{C}{\pi^{2s}}
			\sum_{k\neq k_0, k\geq 0}  \f{\max(k, \si)^{2(1-s)}}{|k-\si|^2} f_k^2. 
		\end{eqnarray*}
	We split the sum in two pieces, $k\in (\si/2, 2\si)$ and the rest. We have 
		\begin{eqnarray*}
			\sum_{k\neq k_0, k\geq 0: k\in (\si/2, 2\si)}  \f{\max(k, \si)^{2(1-s)}}{|k-\si|^2} f_k^2\leq C_s \si^{2(1-s)} 
				\sum_{k\neq k_0, k\geq 0: k\in (\si/2, 2\si)}  \f{1}{|k-\si|^2} f_k^2 \leq C_s \la^{\frac{2}{s}-2} \|f\|_{L^2}^2, 
		\end{eqnarray*}
	since in this case $\max(k, \si)\leq 2\si$ and $\si\sim \la^{\frac{1}{s}}$.  
	
	In the other case, that is $k\leq \si/2\  \textup{or} \  k\geq 2\si$, we have that $|k-\si|\sim \max(k, \si)$, so  
	 \begin{eqnarray*}
	 	\sum_{k\neq k_0, k\geq 0: k\leq \si/2\  \textup{or} \  k\geq 2\si}  \f{\max(k, \si)^{2(1-s)}}{|k-\si|^2} f_k^2 \leq  \sup_{k\leq \si/2\  \textup{or} \  k\geq 2\si}	 	\f{1}{\max(k, \si)^{2s}} \|f\|_{L^2}^2 \leq \frac{1}{\la^2} \|f\|_{L^2}^2.
	 \end{eqnarray*}
	The estimate in this case is exceptionally good, but this is just a small piece of the sum. In all cases, 
	we conclude  \eqref{eq4}.

	Next, we estimate 
	\begin{eqnarray*}
		\int_{-\de}^\de |u(x)|^2 dx &=& \int_{-\de}^\de |u_{k_0} \cos(\pi k_0 x)+\tilde{u}(x) |^2 dx\\ &=& 
		2 |u_{k_0}|^2 \int_0^\de \cos^2(\pi k_0 x) dx+    
		2 \int_{-\de} ^\de u_{k_0} \cos(\pi  k x) \tilde{u}(x) dx + \int_{-\de}^\de |\tilde{u}(x)|^2 dx\\
		&\geq & 
		|u_{k_0}|^2 \de(1+\f{\sin(2\pi k_0 \de)}{2\pi k_0 \de}) - C |u_{k_0}| \|\tilde{u}\|_{L^2}. 
	\end{eqnarray*}
	Note  $(1+\f{\sin(2\pi k_0 \de)}{2\pi k_0 \de})>1-\f{2}{\pi}$, so we can bound   from below
	
	\[
	\int_{-\de}^\de |u(x)|^2 dx \geq \f{\de(1-\f{2}{\pi})}{2} u_{k_0}^2 - C\|\tilde{u}\|_{L^2}^2\geq C_\de u_{k_0}^2 - \f{C}{\la^{2-\frac{2}{s}}} \|f\|^2.
	\]
	Thus, 
	\begin{eqnarray}\label{eqn5}
	u_{k_0}^2 \leq C_{\de}\left(<\lambda>^{\f{2}{s}-2}\|f\|_{L^2}^2+\|u\|_{L^2[-\de,\de]}^2\right).
	\end{eqnarray}
	Hence by combining the estimates \eqref{eq4} and \eqref{eqn5}  , we get
	\[
	\|u\|_{L^2[-1,1]}\leq  C_\de\left(<\la>^{\f{1}{s}-1} \|f\|+ \|u(x)\|_{L^2[-\de,\de]}\right).
	\]
	Lastly, let $-\f{\pi^s}{2^s}<\la<\f{\pi^s}{2^s}$. In this case, we applied  the same arguments as above on
	\[
	u=u_0+\sum_{k=1}^\infty u_k \cos(\pi k x)
	\]
	to get
$	\|\tilde{u}\|_{L^2}  \leq C \|f\|_{L^2}$, while 
	$
	|u_0|^2\leq C_\de \left( \int_{-\de}^\de |u(x)|^2 dx+  \|f\|^2\right).
$
	Finally, we conclude that in all three cases, 
	\[
	\|u\|_{L^2[0,1]}\leq C_\de\left(<\la>^{\f{1}{s}-1} \|f\|_{L^2}+\|u\|_{L^2[-\de, \de]}\right).
	\]
{\bf Case II: $u, f$ are odd functions}
	For $u, f$ odd functions, we can expand $u$ and $f$ in sine series as follows
	\[
	u=\sum_{k=0}^\infty a_k \sin(k\pi x),~ f=\sum_{k=0}^\infty f_k \sin(k\pi x)
	\]
	Again, for $\la<-\f{\pi^s}{2^s}$,  we have the estimate (same as above)
	\[
	\|u\|\leq \f{C}{|\la|} \|f\|. 
	\]
	For $\la=\pi^s \si^ss, ~\si>\f{1}{2}$, we have (same as above in \eqref{eq4})
	\begin{eqnarray*}
		\|\tilde{u}\|_{L^2}\leq C<\la>^{\frac{1}{s}-1} \|f\|.
	\end{eqnarray*}
	where in this case $\tilde{u}=\sum_{k\neq k_0} u_k \sin(\pi  k x)=u-u_{k_0} \sin(\pi  k_0 x)$. 
	Next, we estimate 
	\begin{eqnarray*}
		\int_{-\de}^\de |u(x)|^2 dx &= &\int_{-\de}^\de |u_{k_0} \sin(\pi k_0 x)+\tilde{u}(x) |^2 dx\\& =& 
		2 |u_{k_0}|^2 \int_0^\de \sin^2(\pi k_0 x) dx+   
		2 \int_{-\de} ^\de u_{k_0} \sin(\pi  k x) \tilde{u}(x) dx + \int_{-\de}^\de |\tilde{u}(x)|^2 dx \\
		&\geq & 
		|u_{k_0}|^2 \de(1-\f{\sin(2\pi k_0 \de)}{2\pi k_0 \de}) - C |u_{k_0}| \|\tilde{u}\|_{L^2}. 
	\end{eqnarray*}
	Now, observe $z\to \f{sin(z)}{z}$ can be close to $1$, but in any case, we have 
	$$
	(1-\f{\sin(2\pi k_0 \de)}{2\pi k_0 \de})\geq c \min(1, (k_0\de)^2)\geq c\de^2. 
	$$
	Note that in this last estimate, we used $k_0\geq 1$, so  $c$ is independent on $k_0$! 
Consequently, 
	\[
	\int_{-\de}^\de |u(x)|^2 dx\geq c\de^3 |u_{k_0}|^2- C |u_{k_0}| \|\tilde{u}\|_{L^2}\geq c\de^3 |u_{k_0}|^2- 
	C_\de \|\tilde{u}\|_{L^2}^2 \geq c\de^3 |u_{k_0}|^2 - \f{C_\de}{\la^{2-\f{2}{s}}}\|f\|^2.
	\]
	Hence, 
	$$
	\|u\|_{L^2[-1,1]}^2\leq 2(u_{k_0}^2+ \|\tilde{u}\|_{L^2}^2)\leq C_\de\left(<\la>^{\f{2}{s}-2} \|f\|^2+\int_{-\de}^\de |u(x)|^2 dx\right).
	$$
{\bf  Case III $u, f$ are arbitrary functions}
	In this case, we split $u$ and $f$ in even and odd parts and derive estimates for each of them. Putting it all together, we get
	\begin{eqnarray*}
		\|u\|_{L^2[-1,1]}^2 &=& \|u_{even}\|_{L^2[-1,1]}^2+\|u_{odd}\|_{L^2[-1,1]}^2 \\
		&\leq & C_\de
		\left(\f{\|f_{even}\|^2+\|f_{odd}\|^2}{|\la|^{2-\f{2}{s}}}+ \int_{-\de}^\de (u^2_{even}(x) + u^2_{odd}(x)) dx\right)\\
		&=& C_\de\left(\f{\|f\|^2}{\la^{2-\f{2}{s}}}+ \int_{-\de}^\de u^2(x) dx\right).
	\end{eqnarray*}
	Hence,
	\[
	\|u\|_{L^2[-1,1]}\leq  C_\de\left(\la^{\f{1}{s}-1}\|f\|_{L^2[-1,1]}+ \|u\|_{L^2[-\de,\de]}\right)
	\]
\end{proof}
\noindent This finishes the proof of the observability estimate \eqref{20}. 
Next, we extend Lemma \ref{theo:10} to the whole line $\mathbb{R}$ by using a technique similar to Wunsch, \cite{wunsch2015periodic}.
\subsubsection{Observability on intervals implies observability for a $H_\al$}

 Introduce the   operators
 \[
{\color{red} H_{\alpha}^s:=}[(-i\p_{x}-\alpha)^{2}]^{s/2}~~\text{for}~\alpha\in\mathbb{R}.
 \]
 Equivalently, one may define $H_\al$ through the Fourier transform  
 $$
 \widehat{H^s_\al f}(k)=|k-\al|^s \hat{f}(k).
 $$
Observe the relation 
$$
(-i\p_{x}-\alpha)^{2}=e^{i \al \cdot}(-\p_{xx})  e^{- i \al \cdot}.
$$
Since multiplication by $e^{\pm i \al x}$ is an unitary operator on $L^2[-1,1]$, the relation above is an unitary equivalence between $(-i\p_{x}-\alpha)^{2}$ and $-\p_{xx}$. Consequently, $H^s_{\alpha}$ is a self-adjoint operator, so by Stone theorem, $iH^s_{\alpha}$ generates  a $C_0$-group of unitary operators on a Hilbert space, which we denote  by $U_{\alpha}(t)=e^{it H^s_{\alpha}}$. In addition, and since one can define $g(-\p_{xx})$ for very general functions $g$ (for example $C[0, \infty)$), we have 
\begin{equation}
\label{p:10}
g((-i\p_{x}-\alpha)^{2})=e^{i \al \cdot}g(-\p_{xx}) e^{- i \al \cdot}.
\end{equation}
In particular, applying \eqref{p:10} to the functions $t^{s/2}$ and $e^{i t^{s/2}}$, 
\begin{equation}
\label{p:15}
H^s_\al=e^{i \al \cdot}(-\p_{xx})^{s/2} e^{- i \al \cdot}; \ \ e^{it H^s_{\alpha}} = e^{i \al \cdot}e^{i t H^s_0} e^{- i \al \cdot}.
\end{equation}
The observability estimate for $H^s_{\alpha}$ on   flat torus $\mathbb{T}=\mathbb{R}/\mathbb{Z}$ is as follows. 
\begin{lemma}
	\label{obsfortorus}
	Let $\Gamma\subset \mathbb{T}$ be open and non-empty. For all $\alpha\in [0,1)$, we have
	\begin{eqnarray}\label{estimate1}
	(H^s_{\alpha}-\lambda)u=f\Rightarrow~~\|u\|_{L^2(\mathbb{T})}\leq C\left(<\lambda>^{\f{1}{s}-1}\|f\|_{L^2(\mathbb{T})}+\|u\|_{L^2(\Gamma)}\right)
	\end{eqnarray}
	with constants independent of $\alpha|$ and $|\lambda|\geq 1  \in \mathbb{R}$.
\end{lemma}
\begin{proof}  
	
	Note that for  $\alpha=0$, we have 
$
	H^s_0=(-\p_{xx})^{s/2},
$
	and in this case the result is proved in Lemma  \ref{theo:10}.   
	Next, assume $\alpha\neq 0$.
	
By the results in  \cite{miller2005controllability} and since $H^s_{\alpha}$ is a self-adjoint operator,  
the estimate \eqref{estimate1} is equivalent to Schr\"odinger observability for $H^s_{\alpha}$. That is, we need to establish that  for every, non-empty $\omega\subset\mathbb{T}$ and every $T>0$, there exist $C(T,\omega)$ such that 
	\[
	\|f\|_{L^2}^2\leq\int_0^T\|e^{itH^s_\al}f\|_{L^2(\omega)}^2~dt
	\]
	Next, fix a non-empty open set $\omega$. By $H^s_0$-observability, we have  for every  $T> 0$  
	\begin{eqnarray*}
		\|f\|_{L^2}^2 &=& \|e^{-i \alpha x}f\|_{L^2}^2 \leq 
		C\int_0^T\|e^{it H^s_0}[e^{-i \alpha \cdot}f]\|^2_{L^2(\omega)}dt=\\
		&=& C\int_0^T\| e^{i \al \cdot} e^{it H^s_0} e^{-i \alpha \cdot}f\|^2_{L^2(\omega)}dt  
	= C\int_0^T\|e^{itH^s_{\alpha}}f\|^2_{L^2(\omega)}dt. 
	\end{eqnarray*}
	This proves the Schr\"odinder observability, with the same constants as $\al=0$. 
	Hence by Theorem 5.1 of  Miller \cite{miller2005controllability} , the estimate (\ref{estimate1}) holds for all $s>0$. 
\end{proof}
\subsubsection{Observability for $H_\al$ implies observability} 
For $g\in \langle x\rangle^{-s} H^{-\infty}(\mathbb{R})$ with $s> 1$. We define the periodization of $g$ as follows
\[
\Pi g(x)=\sum\limits_{n\in \mathbb{Z}}g(x+2\pi n).
\]
Also, for $\alpha\in\mathbb{R}$, we set
\[
\Pi_{\alpha}g=\Pi(e^{i\alpha x}g)
\]
\begin{lemma}\label{lemma1.3.4} For $g\in \langle x\rangle^{-s} H^{-\infty}(\mathbb{R})$ with $s> 1$, we have 
	\begin{eqnarray}
	\|g\|_{L^2(\mathbb{R})}^2=\int_{[0,1)}\|\Pi_{\alpha}g\|_{L^2(\mathbb{T})}^2~d\alpha.
	\end{eqnarray}
	Moreover, if $\Omega\subset\mathbb{R}$ is $2\pi\mathbb{Z}$-invariant and $\Omega_0$ denotes its projection to $\mathbb{T}$, we have 
	
	\begin{eqnarray}
	\|g\|_{L^2(\Omega)}^2=\int_{[0,1)^2}\|\Pi_{\alpha}g\|_{L^2(\Omega_0)}^2~d\alpha.
	\end{eqnarray}
\end{lemma}
\noindent For the proof of the lemma, we refer to Lemma 5,  \cite{wunsch2015periodic}.

Note that   $((-\p_{xx})^{s/2}-\lambda)u=f$ implies 
\[
e^{ i \al x}((-\p_{xx})^{s/2}-\lambda)e^{- i \al x}[e^{i \al x} u]=e^{ i \al x} f
\]
In terms of the operator  $\Pi$, we get $	(H_{\alpha}-\lambda)(\Pi_{\alpha}u)=\Pi_{\alpha}f$. 
By Lemma (\ref{obsfortorus}), we conclude 
\begin{eqnarray*}
	\|\Pi_{\alpha}u\|_{L^2(\mathbb{T})}^2\leq C(<\lambda>^{\f{2}{s}-2}\|\Pi_{\alpha}f\|_{L^2(\mathbb{T})}^2+\|\Pi_{\alpha}u\|_{L^2(\Omega_0)}^2)
\end{eqnarray*}
By Lemma \ref{lemma1.3.4}, we may integrate both sides over the set $[0,1)$ to obtain
\begin{eqnarray*}
	\|u\|_{L^2(\mathbb{R})}^2\leq C(<\lambda>^{\f{2}{s}-2}\|f\|_{L^2(\mathbb{R})}^2+\|u\|_{L^2(\Omega)}^2)
\end{eqnarray*}
 This is of course \eqref{eqn:20} and so the proof of Theorem \ref{theo:2} is complete.

\subsection{Resolvent estimate}

From the observability estimate above, we prove the following  resolvent estimate for our damped problem.   
\begin{proposition}\label{thm:1.1}
	Assume that $m>0$, $  \ga(x)\geq 0$ and $\ga \in L^\infty$ and there exist $\epsilon>0$ and a $2\pi \mathcal{Z}$- invariant set $\Omega\in\mathbb{R}$ such that $\ga(x)\geq \epsilon$ for a.e. $x\in\mathbb{R}.$ For the equation
	\begin{eqnarray}\label{eqn1.9}
	((-\p_{xx})^{s/2}+m+ik\ga(x)-k^2)u=f
	\end{eqnarray}
	we have the following:
	\begin{itemize}
		\item For $0<s<2~~~$, 
		\begin{equation}
		\|u\|_{L^2(\mathbb{R})}\leq C<k>^{\f{4}{s} - 3}\|f\|_{L^2(\mathbb{R})}, 
		\end{equation}
		\item For $s\geq 2~~~$, 
		\begin{equation}
		\label{30} 
		\|u\|_{L^2(\mathbb{R})}\leq C <k>^{\f{2}{s}-2 }\|f\|_{L^2(\mathbb{R})}.
		\end{equation}
	\end{itemize}
\end{proposition}

\begin{proof}
	We begin by pairing  the equation \eqref{eqn1.9} with $u$, taking the real part and using Cauchy inequality. For $|k|\leq k_0= \sqrt{m}/2$, we get
	\begin{eqnarray*}
		\|u\|_{H^{s/2}(\rone)}^2+(m-k^2)\|u\|_{L^2(\rone)}^2&\leq &\|f\|_{L^2(\rone)} \|u\|_{L^2(\rone)}\leq  \frac{\|f\|_{L^2(\rone)}^2}{4(m-k^2)} +(m-k^2)\|u\|_{L^2(\rone)}^2
	\end{eqnarray*}
	This implies that
	\begin{eqnarray*}
		\|u\|_{H^{s/2}(\rone)}
		\leq C\|f\|_{L^2(\rone)}. 
	\end{eqnarray*}
	Next we assume that $|k|>k_0$.
	We apply  Theorem  \ref{theo:2}  to equation (\ref{eqn1.9}) with the damping term on the right-hand side and $\lambda=k^2-m$. Noting that $<\la>\sim <k>^2$,  we get 
	\begin{eqnarray}\label{eqn1.203}
	\|u\|_{L^2(\rone)}&\leq& C\left(<k>^{\f{2}{s}-2}\|f\|_{L^2(\rone)}+<k>^{\f{2}{s}-1}\|\ga(x)u\|_{L^2(\rone)}+
	\|u\|_{L^2(\Omega)}\right). 
	\end{eqnarray}
Choose $\Omega$ to be contained in the set where $\ga\geq \epsilon$ a.e. for some $\epsilon>0$.  We obtain 
	$$
	\|u\|_{L^2(\Omega)}\leq \eps^{-1} \|\ga(x)u\|_{L^2(\rone)},
	$$ 
	so \eqref{eqn1.203} becomes 
	\begin{eqnarray}\label{eqn1.20}
	\|u\|_{L^2(\rone)}&\leq& C\left(<k>^{\f{2}{s}-2}\|f\|_{L^2(\rone)}+(<k>^{\f{2}{s}-1}+\eps^{-1}) \|\ga(x)u\|_{L^2(\rone)}\right). 
	\end{eqnarray}
Pairing the equation (\ref{eqn1.9}) with $u$ and taking the imaginary part, we get for $k\geq k_0$, 
	\begin{eqnarray}
	\|\sqrt{\ga(x)}u\|_{L^2(\mathbb{R})}^2\leq\f{C}{<k>}\|f\|\|u\|
	\end{eqnarray}
	Combining these estimates and observing that $\ga\leq C\sqrt{\ga(x)}$ a.e. yields  
	\begin{eqnarray}
	\|u\|_{L^2(\rone)}\leq C\left(<k>^{\f{2}{s}-2}\|f\|_{L^2(\rone)}+ \f{(<k>^{\f{2}{s}-1}+\eps^{-1})}{<k>^{1/2}} \|f\|_{L^2(\rone)}^{1/2}\|u\|_{L^2(\rone)}^{1/2}\right)
	\end{eqnarray}
	Applying Cauchy-Schwarz, we obtain
	\begin{eqnarray}
	\|u\|_{L^2(\rone)}&\leq& C (<k>^{\f{2}{s}-2} + <k>^{\f{4}{s}-3}+<k>^{-1})  \|f\| 
	\end{eqnarray}
By analyzing the cases $s\in (0,2)$ and $s\geq 2$ separately (here $k$ is large), we finally conclude  
	\begin{eqnarray*}
	\|u\|_{L^2(\rone)}  &\leq &  C <k>^{\f{4}{s}-3 }\|f\|_{L^2(\rone)}, \ \ s\in(0,2) \\
	\|u\|_{L^2(\rone)}  &\leq &  C <k>^{\f{2}{s}-2 }\|f\|_{L^2(\rone)}, \ \ s\geq 2
	\end{eqnarray*}
	This completes the proof. 
\end{proof}

\section{Resolvent estimates and proof of Theorem \ref{thm1.1}}
We begin by recasting (\ref{eq1.1}) as an abstract Cauchy problem. Define $U=(u,u_t)^{T}$, then equation (\ref{eq1.1}) can be written as a dynamical system:
 \[
 U_t=\mathcal{A}U
 \]
 where
 $$
 	\mathcal{A}=\left(\begin{matrix}
 		0&I\\-(-\p_{xx})^{s/2}-m&-\ga(x)\end{matrix}\right),
$$
where  we take $D(\ca)=H^s(\rone)\times H^{s/2}(\rone)$.  The basic Hilbert space  is $\mathcal{H}=H^{s/2}(\rone)\times L^2(\rone)$. The fact that $\ca$ generates a semigroup, under this setup,  is standard. 

Next, we  compute  the resolvent of the operator $\mathcal{A}$. 
	Let $u=(u_1,u_2)^{\prime}$ and $ f=(f_1,f_2)^{\prime}$.  Then 
	\[
	(ikI-\mathcal{A})u=f
	\]
	is equivalent to 
	\begin{eqnarray*}
		iku_1-u_2=f_1\\
		((-\p_{xx})^{s/2}+m)u_1+(ik+\ga(x))u_2=f_2
	\end{eqnarray*}
or 
	\begin{eqnarray*}
		u_1&=&((-\p_{xx})^{s/2}+m+ik\ga(x)-k^2)^{-1}\left((ik+\ga(x))f_1+f_2\right)\\
		u_2&=&i k u_1 - f_1.
	\end{eqnarray*}
	Hence, the resolvent of $\mathcal{A}$ is 
	\begin{align*}
	R(ik,\mathcal{A}) =\left(\begin{matrix}
	R(ik)(ik+\ga(x))&&R(ik)\\\\i k R(ik) (\ga(x)+ik)-I&&i k R(ik) \end{matrix}\right), 
	\end{align*}
	where $R(ik)=((-\p_{xx})^{s/2}+m+ik\ga(x)-k^2)^{-1}$.  Note that 
	$$
	R(ik)^*=R(-ik).
	$$

Recall that our basic resolvent estimate,  Proposition \ref{thm:1.1}, provides bounds for the resolvent 
$R(i k)$, acting as operators on $L^2(\rone)$ into itself.  On the other hand, $R(ik)$ are smoothing operators. 
The next result allows us to obtain bounds between different Sobolev spaces. 
\begin{proposition}
	\label{theo:1.2}
	Let $0<s<2$. Then, 
	\begin{equation}
	\label{eqn1.25}
	\|R(i k)\|_{L^2\to H^{s/2}} + \|R(i k)\|_{H^{-s/2}\to L^2} \leq   C <k>^{\f{4}{s}-2}. 
	\end{equation} 
For $s\geq 2$, 
	\begin{equation}
	\label{eqn1.255}
		\|R(i k)\|_{L^2\to H^{s/2}} + 	\|R(i k)\|_{H^{-s/2}\to L^2} \leq    C <k>^{\f{2}{s}-1}. 
	\end{equation}
\end{proposition}
\begin{proof}
	Let $u$ be the solution of 
\begin{eqnarray}
\label{eqn1.24}
((-\p_{xx})^{s/2}+m+ik \gamma(x)-k^2)u=f
\end{eqnarray}
where   $f\in L^2$. 
Taking dot product with $u$ in  (\ref{eqn1.24})   and taking the real part yields 
\begin{eqnarray*}
	\langle {(-\p_{xx})^{s/2}u,u}\rangle+(m-k^2)\langle {u,u}\rangle=\text{Re} \langle {f,u}\rangle\\
	\|u\|_{H^{s/2}}^2\leq \|f\|_{L^2}\|u\|_{L^2}+k^2 \|u\|_{L^2}^2
\end{eqnarray*}
By Proposition \ref{thm:1.1} for $s\in (0,2)$,  $\|u\|_{L^2}\leq C <k>^{\f{4}{s}-3}\|f\|_{L^2}$, so we obtain   
  \begin{eqnarray*}
	\|u\|_{H^{s/2}}^2&\leq&\|f\|_{L^2}\left(<k>^{\f{4}{s}-3}\|f\|_{L^2(\mathbb{R})}\right)+k^2 <k>^{\f{8}{s}-6}\|f\|_{L^2(\mathbb{R})}^2
\end{eqnarray*}
This proves 
$$
\|R(i k)\|_{L^2\to H^{s/2}}\leq  C <k>^{\f{4}{s}-2}, 
$$ and by duality 
$\|R(i k)\|_{H^{-s/2}\to L^2}\leq  C <k>^{\f{4}{s}-2}$. 
For $s\geq 2$, we apply Proposition \ref{thm:1.1} and we similarly obtain  
\begin{eqnarray*}
	\|u\|_{H^{s/2}}^2&\leq&\|f\|_{L^2}\left(<k>^{\f{2}{s}-2}\|f\|_{L^2(\mathbb{R})}\right)+k^2 <k>^{\f{4}{s}-4}\|f\|_{L^2(\mathbb{R})}^2
\end{eqnarray*}
This proves \eqref{eqn1.255}. 
\end{proof}
 Next, we put together the results from Proposition \ref{thm:1.1}, together with Proposition \ref{theo:1.2} to obtain the following result on the composite resolvent $R(i k, \ca)$. 
\begin{proposition}\label{thm1.0}
	For $0<s<2$, there is 
		\begin{equation}
		\label{o:10} 
		\|R(ik,\mathcal{A})\|_{H^{s/2}\times L^2}\leq C <k>^{\f{4}{s}-2}, 
		\end{equation} 
	while for $s\geq 2$ , we have
		\begin{equation}
		\|R(ik,\mathcal{A})\|_{H^{s/2}\times L^2}\leq C.
		\end{equation} 
 
\end{proposition}

\begin{proof}[\bf Proof of Proposition \eqref{thm1.0}]
	
	First we consider the   case $0<s<2$.  
Write  $R(ik,\mathcal{A})$ as follows 
 \begin{eqnarray*} 
 	\left\|R(ik,\mathcal{A})\left(\begin{matrix} f\\g\end{matrix}\right)\right\|_{H^{s/2}\times L^2}
 	&=&\|R(ik)(ik+\ga(x))f\|_{H^{s/2}}+\|R(ik)g\|_{H^{s/2}}+\\&&\|(ik R(ik) (\ga(x)+ik)-I)f\|_{L^2}+
 	\|ik R(ik) g\|_{L^2}
 \end{eqnarray*}
 The estimates for the terms involving $g$ follow easily from the established estimates. Indeed, 
 from \eqref{eqn1.25}, we have 
 $$
 \|R(ik)g\|_{H^{s/2}}\leq C <k>^{\f{4}{s}-2} \|g\|_{L^2}, 
 $$
 while from \eqref{20}, we have 
 $$
 	\|ik R(ik) g\|_{L^2}\leq C |k| <k>^{\f{4}{s}-3} \|g\|_{L^2} \leq  C <k>^{\f{4}{s}-2} \|g\|_{L^2}. 
 $$
 So, it remains to establish the bounds 
 \begin{eqnarray}
 \label{50} 
 & & \|R(is)(ik)(ik+\ga(x))\|=O(|k|^{\f{4}{s}-2}):H^{s/2}\rightarrow H^{s/2} \\
 & & \label{60} 
 \|R(ik)(ik)(\ga(x)+ik)-I)\|=O(|k|^{\f{4}{s}-2}):H^{s/2}\rightarrow L^2. 
 \end{eqnarray}
 Once, \eqref{50} and \eqref{60} are established, we conclude 
 $$
 	\left\|R(ik,\mathcal{A})\left(\begin{matrix} f\\g\end{matrix}\right)\right\|_{H^{s/2}\times L^2}\leq 
 	C \| \left(\begin{matrix} f\\g\end{matrix}\right)\|_{H^{s/2}\times L^2}, 
 $$
 and Proposition \ref{thm1.0} will be proved. 
 
 Next, we estimate $R(ik)(ik)[\ga(x)+ik)]-I:H^{s/2}\rightarrow L^2$.  Elementary manipulations show that  
 \begin{equation}
 \label{80} 
  R(ik)(ik)[\ga(x)+ik)]-I= - R(ik)((-\p_{xx})^{s/2}+m)
 \end{equation}

 Combining \eqref{eqn1.25},  together with the fact that $(-\p_{xx})^{s/2}:H^{s/2}\rightarrow H^{-s/2}$ is continuous, we obtain for $f\in H^{s/2}(\mathbb{R})$ 
 \begin{eqnarray*}
 	&&\|(R(ik)(ik)[\ga(x)+ik)]-I)f\|_{L^2} = \|R(i k)((-\p_{xx})^{s/2}+m)f\|_{L^2} \leq \\
 	& & \leq C|k|^{\f{4}{s}-2}\|((-\p_{xx})^{s/2}+m)f\|_{H^{-s/2}}
 	\leq  C|k|^{\f{4}{s}-2}\|f\|_{H^{s/2}}
 \end{eqnarray*}
 This proves \eqref{60}.

 It remains to estimate  $ \|R(is)(ik +\ga(x))\|_{H^{s/2}\rightarrow H^{s/2}}$. A variant of \eqref{80}reads  
 $$
 R(i k)(i k +\ga(x))=\frac{1}{i k}[I - R(i k)((-\p_{xx})^{s/2}+m)],
$$
Let  $u=R(ik)((-\p_{xx})^{s/2}+m)f$, then
$$
((-\p_{xx})^{s/2}+m+ik\ga(x)-k^2)u=((-\p_{xx})^{s/2}+m)f 
$$
 Pairing  this  equation   with $u$  and taking   real 
 parts and applying Cauchy-Schwarz, we get,
 \begin{eqnarray*}
 	\|(-\p_{xx})^{s/4}u\|_{L^2}^2-(k^2-m)\|u\|_{L^2}^2&\leq& \|((-\p_{xx})^{s/2}+m)f\|_{H^{-s/2}}\|u\|_{H^{s/2}}\\
 	&\leq& \|f\|_{H^{s/2}}\|u\|_{H^{s/2}}.
 \end{eqnarray*}
 Therefore,
 \begin{eqnarray}\label{eqn1.16}
 \|u\|_{H^{s/2}}^2\leq C(k^2 \|u\|_{L^2}^2+ \|f\|_{H^{s/2}}^2).
 \end{eqnarray}
 Next, when we estimate  $\|u\|_{L^2}$,    we used \eqref{eqn1.25}  to get 
 \begin{eqnarray*} 
 \|u\|_{L^2} &=&  \|R(ik)((-\p_{xx})^{s/2}+m)f\|_{L^2}\leq C|k|^{\f{4}{s}-2}\|(-\p_{xx})^{s/2}+m)f\|_{H^{-s/2}} \leq \\
 &  \leq &  C|k|^{\f{4}{s}-2}   \|f\|_{H^{s/2}}
 \end{eqnarray*}
 Plugging this estimate back in \eqref{eqn1.16}, we obtain 
 $
 \|u\|_{L^2}\leq C |k|^{\f{4}{s}-1} \|f\|_{H^{s/2}}. 
 $
As a consequence, 
 \[
 R(ik)((-\p_{xx})^{s/2}+m)=O(|k|^{\f{4}{s}-1}):H^{s/2}(\mathbb{R})\rightarrow H^{s/2}(\mathbb{R}),
 \]
 whence for large $|k|$, 
 \begin{eqnarray*}
 \|R(i k)(i k +\ga(x))\|_{H^{s/2}\to H^{s/2}} &=& k^{-1}\|I- R(ik)((-\p_{xx})^{s/2}+m)\|_{H^{s/2}\to H^{s/2}} \leq \\
 &\leq &  Ck^{-1}(1 + |k|^{\f{4}{s}-1})\leq C |k|^{\f{4}{s}-2},
 \end{eqnarray*}
 which is \eqref{50}.  
Hence, for $0<s<2$, we get 
$$
R(ik, \ca)=(ik-\mathcal{A})^{-1}=O(|k|^{\f{4}{s}-2}): H^{s/2}\times L^2\rightarrow H^{s/2}\times L^2. 
$$ \
Similarly, for $s\geq 2$, we have
$$
R(ik, \ca)=(ik-\mathcal{A})^{-1}=O(|k|^{\f{2}{s}-1}): H^{s/2}\times L^2\rightarrow H^{s/2}\times L^2
 $$
 So, in fact, we have decay in $k$ of the resolvent for $s>2$. 
\end{proof}
 Having proved Proposition \ref{thm1.0}, we are ready for the proof of our main result, Theorem \ref{thm1.1}. 
 For the case $0<s<2$, we apply the  Borichev-Tomilov Theorem \ref{mainthm1} with $\al=\f{4}{s}-2>0$. Then, the semigroup satisfies the following bound 
 $$
 \|e^{t \ca} (\mu-\ca)^{-1}\|_{H^{s/2}\times L^2\to H^{s/2}\times L^2}\leq C t^{-\f{s}{4-2 s}},
 $$
 for any $\mu\in \rho(\ca)$, say $\mu=1$. Equivalently, 
 $$
 \|e^{t\ca} f\|_{H^{s/2}\times L^2}\leq C t^{-\f{s}{4-2 s}} \|(1-\ca) f\|_{H^{s/2}\times L^2}\leq C t^{-\f{s}{4-2 s}}\|f\|_{H^s\times H^{s/2}},
 $$
 since $\ca:H^s\times H^{s/2}\to H^{s/2}\times L^2$. 
 
For  $s\geq 2$,  by Gearhart-Pr\"uss Theorem \ref{Gearhart} the  energy of the damped fractional Klein-Gordon is decaying exponentially and more precisely, we have the bound \eqref{s2}.

The authors are thankful to the anonymous  referee  and to Reinhard Stahn for their  useful comments, which helped to improve the manuscript.


%

\end{document}